\documentclass[14pt]{article}
\usepackage[english]{babel}
\usepackage{amsfonts,amssymb,amsmath,amstext,amsbsy,amsopn,amscd,amsthm,graphicx,euscript,graphicx}
\usepackage{graphics}
\newtheorem{theorem}{Theorem}
\newtheorem*{theorem*}{Theorem}

\newtheorem{proposition}{Proposition}
\newtheorem{corollary}{Corollary}

\newtheorem{remark}{Remark}

\newcounter{tdfn}
\newenvironment{dfn}
{\vspace{0.15cm}{\bf Definition \arabic{tdfn}.}} {\par
\addtocounter{tdfn}{1}}
\newcounter{trk}
{\endtrivlist}

\def\:{\colon}

\def\R{{\mathbb R}}
\def\Z{{\mathbb Z}}
\def\0{{\mathbf 0}}
\def\1{{\mathbf 1}}

\def\R{{\mathbb R}}

\def\N{{\mathbb N}}
 
\newcommand{\skcrossr}{\raisebox{-0.25\height}{\includegraphics[width=1.5em]{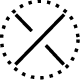}}}

\newcommand{\skcrf}{\raisebox{-0.25\height}{\includegraphics[width=1.5em]{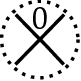}}}
\newcommand{\skcrv}{\raisebox{-0.25\height}{\includegraphics[width=1.5em]{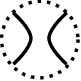}}}
\newcommand{\skcrh}{\raisebox{-0.25\height}{\includegraphics[width=1.5em]{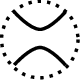}}}

\theoremstyle{definition}
\newtheorem{definition}{Definition}
\newtheorem{example}{Example}
\newcommand{\oast}{\overline{\ast}}
\newcommand{\uast}{\underline{\ast}}

\author{V.O.Manturov \footnote{Moscow Institute of Physics and Technology, Moscow 141700, Russia, Nosov Magnitogorsk State Technical University, Zhilyaev Laboratory of mechanics of gradient nanomaterials, 38 Lenin prospect, Magnitogorsk, 455000, Russian Federation}, I.M.Nikonov\footnote{Lomonosov Moscow State University}}

\title{Flat-virtual knot: introduction and some invariants}

\date{}

\begin{document}

\maketitle

\begin{abstract}

The motivation for this work is to construct a map from classical knots to virtual ones. What we get in the paper is a series of maps from knots in the full torus (thickened torus) to flat-virtual knots. We give definition of flat-virtual knots and presents Alexander-like polynomial and (picture-valued) Kauffman bracket for them.

\end{abstract}



\section{Introduction}

In the paper \cite{VirtualKnots} L.H.Kauffman introduced {\em virtual knot theory}
which turned out to be a generalisation of classical knot theory from the following points of view:

\begin{enumerate}

  \item Diagrammatically, virtual knot (and link) diagrams admit a new type of crossings (depicted by an encircled
  crossing) called {\em virtual} where the equivalence is given by {\em classical moves} dealing with only classical
  crossings and the {\em detour move} obtained by deleting a piece of curve connecting two points and containing only virtual crossings and
  virtual self-crossings inside, and redrawing it in another place.

  \item From the Gauss diagram point of view, classical knots are equivalence classes of rather special Gauss diagrams (corresponding to planar
  diagrams) modulo Reidemeister moves. Virtual knots, in turn, are equivalence classes of {\em all Gauss diagrams} modulo the same moves
  (we need not make artificial restrictions, hence virtual knots are more natural than classical ones).

\item Classical knots are knots in the thickened plane (thickened sphere); if one takes oriented thickenings of oriented $2$-surfaces of higher genera, then
    one gets to the theory of virtual knots: a {\em virtual link} is an equivalence class of links in thickened surfaces modulo {\em isotopy} of links inside the thickened surface,  {\em stabilisation}, and {\em destabilisation}.
    \footnote{By {\em stablisation} of a knot in $S_{g}\times I$ we mean a knot in $S_{g+1}\times I$ obtained by
    adding a handle to $S_{g}$ missing the knot projection; a {\em destabilisation} is the inverse operation}.

\end{enumerate}

A nice theorem due to G.Kuperberg \cite{Kup} (proven in many other ways, see, e.g., \cite{ClProj})
says that {\em minimal genus representative of a virtual link is unique}, i.e., if two virtual links $L_{1},L_{2}$
in thickened surfaces can be obtained from each other by isotopies, stabilisations, and destabilisation, then we can
destabilise them both to the minimal genus representative\footnote{For split links, the minimal genus surface might
be disconnected.} where the resulting links will be isotopic.





Along with virtual knots, one considers {\em flat knots}
\cite{Book},
{\em free knots}, \cite{Parity}
and other generalisations and simplifications.

{\em Flat knots} are obtained from virtual knots by forgetting over/under information at each crossing, i.e.,
we are left with isotopy classes of curves on surfaces without thickenings. For a fixed (say, minimal) genus
such objects are just free isotopy classes of curves.
These objects are widely studied and for flat knots there is a ``gradient descent algorithm'' which allows one to
get to the minimal representative (defined up to third Reidemeister moves) by using only decreasing moves and
destabilisations. This phenomenon is related to the {\em curve-shortening algorithm},
\cite{HassScott}, where the process of decreasing the length is closely related to the process of decreasing the
crossing number.


Initially, free knots appeared under the name of {\em isotopy classes of Gauss words} introduced by V.G.Turaev \cite{Tur},
where he conjectured to be trivial, but in \cite{Parity} the author not only showed non-triviality
of some free knots but also constructed {\em picture-valued invariants} (the parity bracket).
Unlike many well-known invariants which are valued in numbers, groups, modules, homology groups etc.,
the {\em parity bracket} and other similar invariants are valued in {\em linear combinations of knot diagrams}.
Defined on such a simple object as free knots, the parity bracket is strong enough to realise the following principle~\cite{book}:
\begin{quote}
    \em if a diagram is complicated enough then it realises itself.
\end{quote}

Without giving explicit definitions, we
suggest the reader just to think of free groups: the reduced form of a word appears inside any (non-reduced) word equivalent to it.
In other words, {\em local minimality yields global minimality}.
For standard knot theory, this is generally not the case: there are non-trivial diagrams of the unknot which cannot be reduced.

For flat virtual knots (objects which are much simpler than virtual knots themselves but nevertheless interesting enough for
the effects we are interested in) there is a parity bracket $K\mapsto [K]$ which associates to a flat knot $K$ a $\Z_{2}$-linear combination
$[K]$ of diagrams (may be up to some easy moves). The only two things we need to know are:
\begin{enumerate}
\item For a concrete knot diagram $K$ all summands in $[K]$ are obtained from $K$ by some {\em smoothings} (a sort of simplifications);
\item There are lots of diagrams $K$ (irreducibly odd ones) for which one has $$[K]=K,$$ i.e., for the flat knot represented by the
diagram $K$ the value of the bracket is equal to the diagram $K$ with coefficient $1$.
\end{enumerate}
These two facts lead to the following: if some diagram $K'$ is equivalent to $K$ then $[K']=K$, hence $K$ ``sits inside $K'$'',
i.e., we get the desired phenomenon:

For a diagram $K$ which is {\em complicated enough} (odd and irreducible) any diagram equivalent to it ($K'$)
``contains it as a subdiagram'' (i.e., $K$ can be obtained from $K'$ by some smoothings).

This parity bracket is based on the notion of {\em parity} which can be defined {\em axiomatically} and which is always trivial
for classical knots (in our context this means that for classical knots there are no ``odd'' diagrams, so $[K]=K$ never occurs).

Actually, picture-valued invariants and formulas like $[K]=K$ appear in some other contexts of virtual knot theory~\cite{KM1,KM2}, but in all those
theories one requires some {\em non-trivial homology of the ambient space}, i.e., to get ``nice pictures'', one has to draw them
on a nice 2-surface, the higher the genus of the surface is, the nicer are the pictures.

Knots in the full torus, $S^{1}\times D^{2}$ are not very far from classical ones, and the parity one can get from them does not
lead to very strong consequences. Nevertheless, the aim of the present paper is

{\em to construct a map from knots in the full torus to flat-virtual knots.}

The latter objects (will be defined later in the text) occupy a preliminary position between virtual knots and flat knots.
Roughly speaking, virtual knots are knots in thickenings of $2$-surfaces (up to stabilisation/destabilisation, i.e., handle addition/removal),
and flat knots are curves in $2$-surfaces (without over/undercrossing information). Flat-virtual knots
have some crossings with over/under information and some crossings without such.

Actually, free knots are simplifications of flat knots, so, already for flat knots one can apply the nice parity techniques.

So, the main objects of the present paper, the flat-virtual knots, are complicated enough to bear parity brackets and other
invariants, and, on the other hand, they appear as images of the maps from knots in the cylinder.

Hence, one can ``pull-back'' all the theory of ``picture-valued invariants'' and other virtual features to knots in the full torus.

The present paper deals with a ``mixed theory'' of knots which have {\em classical}, {\em flat}, and {\em virtual}
crossings. If we draw such diagrams on a surface, only classical and virtual crossings will be enough.

The paper is organised as follows. In the next section, we give definitions and prove easiest theorems.
The following three sections are devoted to the demonstration of three ``virtual'' features of flat-virtual knots:
In Section 3, we discuss the Alexander matrix for flat-virtual knot which is ``upgraded'' by adding two new types
of variables corresponding to flat crossings and virtual crossings. These new variables make the determinant of the matrix
non-zero (unlike the classical case).

In Section 4, we define a version of (Kauffman parity) bracket which gives rise to a picture-valued invariant
of flat-virtual knots.

Finally, in Section 5, we deal with cobordisms of flat knots and prove the theorem that if a knot in the thickened cylinder is slice then its image (flat-virtual knot) is slice as well.

This gives rise to sliceness obstruction for classical objects: knots in the full torus $S^{1}\times D^{2}$.

The appendix contains a sketch of theory of multi-flat knots.

\section{Acknowledgements}

The authors extremely grateful to 
Louis Hirsch Kauffman for many fruitful and stimulating discussions and to Darya Popova.

The study was supported by the grant of Russian Science Foundation (No. 22- 19-20073 dated March 25, 2022 “Comprehensive study of the possibility of using self-locking structures to increase the rigidity of materials and structures”).

\section{Flat knots and links: basic definitions}

A {\em flat-virtual link diagram} is a four-valent graph on the plane
where each vertex is of one of the following three types:
\begin{enumerate}
\item classical (in this case one pair of edges is marked as an {\em overcrossing strand});
\item flat;
\item virtual.
\end{enumerate}

The number of link components of a flat-virtual link diagram is just the number of unicursal
component of such. A {\em flat-virtual knot diagram} is a one-component flat-virtual link diagram.

The number of components is obviously invariant under the moves listed below, so, it will be reasonable
to talk about the number of components of a {\em flat-virtual link.}

\begin{dfn}
Here a {\em flat-virtual link} is an equivalence class of flat-virtual link diagrams modulo the following moves:
\begin{enumerate}
\item Classical Reidemeister moves which affect only classical crossings [there are three such moves in the non-orientable
case];
\item Flat Reidemeister moves which refer to flat crossings only and are obtained by forgetting the information about
overpasses and underpasses [the second Reidemeister move and the third Reidemeister move];
\item Mixed flat-classical Reidemeister-3 move when a strand containing two consecutive flat crossings
passes through a classical crossing of a certain type [the over-under structure at the unique classical crossing is preserved];
\item Virtual detour moves: a strand containing only virtual crossings and self-crossings can be removed
and replaced with a strand with the same endpoints where all new intersections are marked by virtual crossings.
\end{enumerate}

\end{dfn}

\begin{remark}
The last move can be expressed by a sequence of local moves containing purely virtual Reidemeister
moves and third mixed Reidemeister moves where an arc containing virtual crossings only goes through
a classical (or a flat) crossing.
\end{remark}

\begin{dfn}
By a {\em restricted flat} virtual link we mean an analogous equivalence class where
we forbid the third Reidemeister move with three flat crossings.

\end{dfn}

\begin{remark}
Following Bar-Natan's point of view, virtual knots (links) can be considered as having only classical crossings
meaning that the Gauss diagram (multi-circle Gauss diagram) indicates the position of crossings on the plane
and the way how they have to be connected to each other (the way the connecting strands virtually intersect each other
does not matter because of the detour move).

The same way, one can define flat-virtual knot and link diagrams can be encoded by Gauss diagrams with
the only difference being that when a classical crossing carries two bits of information (over/undercrossing and
writhe) while the flat crossing only carries one bit of information (clockwise rotation).
\end{remark}

\begin{remark}
There is an obvious forgetful map from flat-virtual links to flat links: we just make all classical crossing
flat and forget the over-under information.
\end{remark}

Define this forgetful map by $\zeta:K \mapsto \zeta(K)$ and call $\zeta(K)$ the {\em flat knot corresponding
to $K$}.

\begin{theorem}
$\zeta$ is a well defined map: if $K,K'$ are equivalent flat-virtual link diagrams then $\zeta(K),\zeta(K')$
are equivalent flat knots.
\end{theorem}
\begin{proof}
    One can see that the forgetful map turns the Reidemeister moves of flat-virtual knots into the moves of flat knots. 
\end{proof}

Flat knots project further to free knots, thus we get a map from flat-virtual knots to free knots.

Given a cylinder $C=S^{1}\times [0,1]$ with angular coordinate $\alpha\in [0,2\pi=0)$ and vertical coordinate $z\in [0,1]$.
Fix a number $d\in \N$ and say that two points $x,y$ on $C$ are {\em equivalent} if they have the same vertical coordinate and
their angular coordinate differs by $\frac{2\pi k}{d}$ for some integer $k$.
Likewise, for the torus $T^{2}$ with two angular coordinates $(\alpha,\beta)\in [0,2\pi=0)$ we fix a lattice $l$
as a discrete subgroup. Say, choosing $d_{1},d_{2}$ our group contains points with angular coordinates
$(\frac{2\pi k_{1}}{d_{1}},\frac{2\pi k_{2}}{d_{2}})$. We say that two points $A,B$ on the torus are {\em equivalent} if their
difference $A-B$ belongs to the subgroup $l$.

Now, having a diagram $K$ on the cylinder (torus) with a fixed number $d$ (lattice $l$) we define
a {\em flat-virtual} diagram $\phi_{d}(K)$ (resp., $\phi_{l}(K)$)
(up to virtual detour moves) as follows. First, we assume without loss of generality that
$K$ is {\em generic with respect to the subgroup}, i.e., for any two distinct equivalent points $e,e'$ belonging
to edges are not crossing points and the tangent vectors at these edges are transverse.

In both cases (for $\phi$ and for $\phi'$) for each pair of equivalent points on the edges
(say, $(e_{j},e'_{j})$) we create a new {\em flat} crossing in the following manner.
Let $e'_{j}= e_{j}+g$, where $g$ is an element of the corresponding group ($\Z_{d}$ or $l$).
We choose one
of these points (never mind which one, say, $e_{j}$), remove a small neighbourhood $N=U(e_{j})$
with endpoints $A,B$, shift it by $g$ and connect $A$ to $A+g$ and $B$ to $B+g$ by any curves on the
plane not passing through any crossings (all crossings on the newborn curves $[A,A+g], [B,B+g]$
are declared to be virtual).

\begin{theorem}
The maps $\phi_{d},\phi_{l}'$ defined above are well defined maps from links on the thickened cylinder (resp., thickened torus)
to flat-virtual links. In other words, if $K$ and $K'$ are isotopic diagrams on the cylinder (torus)
then $\phi_{d}(K)\sim \phi_{d}(K')$
(resp., $\phi'_{l}(K)\sim \phi'_{l}(K)$) where
$\sim$
denotes the equivalence of flat-virtual links.
\end{theorem}

\begin{proof}
    We consider the map $\phi_d$, the proof for $\phi'_l$ is analogous.

    We should check that two diagram $K$ and $K'$ which are connected by a diagram isotopy or a Reidemeister move, give equivalent diagrams $\phi_d(K)$ and $\phi_d(K')$.

    Consider an isotopy $K(t)$ of the diagram $K$. For any $1\le k<d$ consider the diagram $K_k$ which is obtained from $K$ by the angular coordinate shift by $\frac{2\pi k}{d}$, and call it the \emph{$k$-th clone} of $K$. The structure of $\phi_d(K(t))$ changes when a pair of equivalent points (dis)appears in $K$, i.e. a new intersection between $K$ and some its clone (dis)appears.  For a generic isotopy $K(t)$ this event corresponds to a Reidemeister move between the diagram $K(t)$ and a clone $K_k(t)$. There can be one of three occurrences:
    \begin{itemize}
        \item an intersection between an arc of $K(t)$ and an arc of $K_k(t)$ appears. This case correspond to a flat second Reidemeister move;
        \item an arc of the clone pass over a crossing in $K(t)$. This case produces a mixed flat-classical third Reidemeister move;
        \item an arc of another clone $K_{k'}(t)$ pass over an intersection between the diagram $K(t)$ and its clone $K_k(t)$. This case corresponds to a flat third Reidemeister move.
    \end{itemize}

    Thus, an isotopy of $K(t)$ produces a sequence of isotopy, flat second and flat third Reidemeister moves of $\phi_d(K(t))$.

    Assume that $K$ and $K'$ are connected by a Reidemeister move. We can suppose that the place where the move occurs is small enough and distinct from the clones of the diagram $K$. Then one can realize the same move between the diagrams $\phi_d(K)$ and $\phi_d(K')$. Hence, the diagrams $\phi_d(K)$ and $\phi_d(K')$ are equivalent.
\end{proof}

Moreover, for the special case of knots in the full torus for $d=2$ we have a more exact statement
\begin{theorem}
If $K$ and $K'$ are isotopic diagrams on the cylinder then
$\phi(K)$ and $\phi(K')$ are diagrams of the same restricted flats virtual link.
\end{theorem}
\begin{proof}
    We can repeat the reasonings of the previous theorem. The difference is in the case $d=2$ there can not be situations when the diagram $K(t)$ and its two clones $K_k(t)$ and $K_{k'}(t)$ have common point because $K(t)$ has only one clone. Thus, we don't need the flat third Reidemeister move.
\end{proof}

\section{The Alexander-like polynomial}

In the present section, we construct an Alexander-like polynomial. In a way very similar
to \cite{AlexGen} (see also~\cite{Man2002a,Man2002b}) where virtual links were considered, for a diagram having $n$ crossing
we shall consider an $n\times n$-matrix, an analogue of the classical Alexander
matrix. The determinant of this matrix, similarly to the case of virtual knots,
will be non-zero (unlike classical knot case, where it is zero). This determinant
will be an invariant.

Let $K$ be an oriented flat-virtual link diagram. We consider it as an image of a disjoint collection
of circles $S^{1}\sqcup\cdots \sqcup S^{n}$ immersed in the plane; the immersion has to be embedding
outside crossings (classical, virtual, and flat), and at crossings, the two arcs of the circle
intersect transversally. All $S^{j}$ are link components.

Almost without loss of generality, we may assume that each link component passes through
at least one classical {\em underpass}. Having this, we can define {\em long arcs} as follows:
a {\em long arc} is the image of a segment of $S^{1}$ which goes from an underpass to the next
underpass. So, long arcs may contain overpasses, virtual crossings, and flat crossings.
It is obvious that the number of long arcs is equal to the number of crossings: for each
crossing there exist exactly one long arc outgoing from this crossing.
We enumerate crossings from $1$ to $n$ arbitrarily and enumerate long arcs
in such a way that the $i$-th long arc emanates from $i$-th crossing.

Each long arc is subdivided into {\em short arcs} by virtual crossings and flat crossings.

After that, we define {\em labels} of short arcs as follows. The labels
will be monomials of the form $u^{k}v^{l}, k,l\in \Z$ in variables $u,v$.
Each long arc starts with a short arc that we call {\em initial}. We associate
the monomial $1$ with the initial arc. After that, when passing through a
virtual crossing, we multiply the label by $v^{\pm 1}$, and when passing through
a flat crossing, we multiply the label by $u^{\pm 1}$. The exponent is chosen as shown in Fig.\ref{inFig}.

\begin{figure}
\centering\includegraphics[width=100pt]{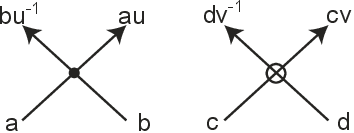}
\caption{Edge labels for flat (left) and virtual (right) crossings}
\label{inFig}
\end{figure}

So, for a short arc $s$, we denote its label by $l(s)\in \Z[v^{\pm 1}, u^{\pm 1}]$.

We define the matrix $A(M)$ to be the $(n\times n)$-matrix whose
entries are defined as follows.

Consider the $i$-th crossing. Assume it is incident to three short arcs,
the initial short arc of the $i$-th arc is emanating, some short arc with
label $l'$ belonging to the the long arc number $i'$ is incoming, and some short arc
with label $l''$ belonging to the long arc number $i''$ forms an overcrossing.

Let the writhe number of this $i$-th crossing be $w(i)$.

Then the $i$-th row of the matrix $A$ contains the elements
$a_{i,i} = -1, a_{i,i'}= l' t^{w(i)}, a_{i,i''}= l'' (1-t^{w(i)}).$

If some of the three labels $i,i',i''$ coincide, we just sum up
the coefficients as above to get the corresponding entry.

Hence we get a matrix $A(K)$ with entries in the Laurent polynomials in $t,v,u$.

Define $\Delta(K) = Det(A(K))$.

We say that two Laurent polynomials, $P$ and $P'$ in $\Z[t^{\pm 1}, v^{\pm 1}, u^{\pm 1}]$ are {\em similar}
(we write $P\doteq P'$ if $P' = \pm  t^{a}v^{b}u^{c}P$ for some integers $a,b,c$).

The main theorem of the present section is the following
\begin{theorem}
Let $K,K'$ be two equivalent flat-virtual link diagram.
Then $\Delta(K)\doteq \Delta(K')$.
\end{theorem}

The proof repeats step-wise the arguments of~\cite[Theorem 18.9]{ManKnotTheory}.

\begin{corollary}
If for a flat-virtual diagram $K$ $\Delta(K)\neq 0$ then the flat-virtual knot $K$ is not classical
(can not be realised by a classical representative).
\end{corollary}


\begin{theorem}
If two classical link diagrams $L,L'$ are equivalent in the category of flat-virtual links
then they are classically isotopic.
\end{theorem}

\begin{proof}
    Consider the map $\mu$ on flat-virtual diagrams which replaces all flat crossings by virtual ones. The map $\mu$ induces a well-defined map from flat-virtual links to virtual links. Hence, the classical diagrams $L=\mu(L)$ and $L'=\mu(L')$ are equivalent as virtual links. By Kuperberg's theorem~\cite{Kup}, the links $L$ and $L'$ are equivalent as classical links.
\end{proof}

\section{A Kauffman bracket (parity bracket) for flat-virtual knots}

The standard Kauffman-bracket version of the Jones polynomial for classical knot works as follows.
Having a classical knot diagram $K$, we take all its crossings and smooth each classical crossing
in one of the two possible ways,
the $A$-smoothing: $\skcrossr\to \skcrh$ and the $B$-smoothing: $\skcrossr\to \skcrv$.
We associate weight $a$ with each $A$-smoothing and weight $a^{-1}$ with each $B$-smoothing.
Having smoothed all classical crossings, we are left with a collection of $d$ curves on
the plane, and we associate $(-a^{2}-a^{-2})^{d-1}$ (or in some version $(-a^{2}-a^{-2})^{d}$)
with this collection of curves (we may think that we associate the factor $(-a^{2}-a^{-2})$ with each factor).

Written compactly, the Kauffman bracket looks like
$$\langle K\rangle = \sum_{s} a^{\alpha(s)-\beta(s)}(-a^{2}-a^{-2})^{\gamma(s)-1},$$
where $s$ runs over all smoothings of classical crossings of the diagram $K$
(there are $2^{n}$ of them if the total number of classical crossings is $n$),
$\alpha(s)$ and $\beta(s)$ for each smoothing denote the number of crossings smoothed
positively (resp., negatively) with respect to $s$, and $\gamma(s)$ denote the resulting
number of circles.

The bracket $\langle K\rangle$ turns out to be invariant under the second and the third Reidemeister
moves, and it gets multiplied by $(-a)^{\pm 3}$ after applying the first Reidemeister move.
The latter can be easily handled by a simple normalisation, which leads to the Jones polynomial
of the classical knot.

In the case of virtual knot we can do the same verbatim if we just disregard virtual crossings:
we smooth all classical crossings as above, and we are left with a system of {\em immersed curves};
we count the curves as before and we are left with the same formula. The invariance proof under
classical Reidemeister-2 and Reidemeister-3 moves is literally the same, and when we apply
the detour move, there is actually nothing to prove. Indeed, we are smoothing classical crossings and
we are interested in the ways how they are connected, so it is not important how the circles are
drawn (embedded) in $\R^{2}$.

A crucial observation which appeared in several papers
(my paper, Miyazawa etc.) was that by looking at {\em states} (results of smoothings) of classical
crossings we can get much more information than just the {\em number} of circle.
Indeed, if we live in a certain (thickened) $2$-surface, we get a collection of curves on this $2$-surfaces
and we can treat the formula
$$\langle K\rangle =\sum_{s} a^{\alpha(s)-\beta(s)} K_{s}$$
in a more meaningful way by looking at $K_{s}$ as something more intricate than just $(-a^{2}-a^{-2})^{\alpha(s)-1}$.

Hence, we distinguish between {\em classical crossings} (which are to be smoothed) and {\em virtual crossings}
(which are left as they are and may give rise to some interesting information). In
\cite{Parity}, the author considered {\em free knots} where no over/under information at classical crossings
was specified, there was no natural way to distinguish between $A$-smoothings and $B$-smoothings and
the only thing which could be left from
$a^{(\alpha(s)-\beta(s))}$ was just $1$, the non-trivial element of the 2-element field.

Nevertheless, the protagonist of \cite{Parity} was the notion of {\em parity}. For free knots we are distinguishing
between {\em even crossings} and {\em odd crossings}. After that, we shall smooth all even crossings in one of the two
ways, and do nothing with odd crossings. This will lead to {\em states}, i.e., graphs whose vertices
are odd crossings (among the initial crossings).
For {\em free knots} (one-component links) one way to define the parity is the {\em Gaussian parity}: we take a Gauss diagram
and for each chord look how many chords are linked with this one\footnote{We say that chords $a,b$ are linked if their end points
alternate}. To have a feeling how parity behaves under Reidemeister moves, one should consider $2$-component links
and say that a pure crossing (formed by one component twice) is even and a {\em mixed} crossing (formed by two different
componentss) is odd. Then one easily gets to {\em parity axioms} which are satisfied by chords of a
Gauss diagram.

Hence, in the parity bracket we have

$$[K]= \sum_{s-odd} K_{s} \in \Z_{2}{\cal G}$$

Here for a diagram with $n$ even crossings and $m$ odd crossings we take $2^{n}$ possible smoothings $s$
at even crossings and the corresponding diagrams $K_{s}$. To make $[K]$ {\em invariant under Reidemeister moves},
one should consider the resulting smoothed diagrams $K_{s}$ modulo some equivalence relations, so, ${\cal G}$
actually consists of equivalence classes of diagrams modulo some moves.

The greatest feature is that when passing from $K$ to $[K]$ almost no moves are left: ${\cal G}$ considered up
to second Reidemeister moves and some simple factorisation.
So, up to some small justification, we can say that $[\cdot]$ takes a free knot to a graph, and,
if $K$ has only odd crossings then by construction $[K]=K$.

Without going much in detail with the parity bracket, let us pass to the {\em virtual-flat bracket} to
be denoted by $\langle K\rangle_{flat})$.

The reader should pay the most attention to the following things:

\begin{enumerate}
\item We smooth classical crossings according to the formula
$\skcrossr\mapsto a \skcrv + a^{-1}\skcrh$

\item We do not smooth any other crossings and take the resulting diagrams
to constitute the corresponding $K_{s}$.

\item The resulting invariant $\langle K\rangle_{flat} = \sum_{s}a^{(\alpha(s)-\beta(s))}K_{s}\in \Z[a,a^{-1}]{\cal G}_{f}$
takes value in $\Z[a,a^{-1}]$-linear combinations of some equivalence classes of diagrams
with flat and virtual crossings.

\end{enumerate}

Briefly, we get rid of all classical crossings by smoothings and define ${\cal G}_{f}$ as follows.
Elements of ${\cal G}_{f}$ are equivalence classes of plane diagrams with two types of crossings, virtual ones
and flat ones. Two diagrams $D_{1},D_{2}$ from ${\cal G}_{f}$ are called {\em equivalent} if $D_{2}$ can be
obtained from $D_{1}$ by applying a sequence of the following moves:
\begin{enumerate}
\item The flat second Reidemeister move
\item The flat third Reidemeister move
\item The detour move which encompasses all ways of erasing and redrawing curves containing virtual
crossings only.
\item $K\sqcup \bigcirc = (-a^{2}-a^{-2}) K$ meaning that the graph obtained from $K$ by adding a separate
empty circle is equal to $K$ multiplied by $(-a^{2}-a^{-2})$.
\end{enumerate}

We can naturally define {\em oriented} flat-virtual links by putting orientation on each component of
a link. Then for a link diagram $L$ we define its (classical) writhe number $w(L)$ by counting the writhe numbers
of classical crossings algebraically.  Obviously, $w(L)$ does not change when applying all Reidemeister
moves except the first classical one, and changes by $\pm 1$ when applying the first Reidemeister move.

\begin{theorem}
The flat-virtual bracket $K\mapsto \langle K\rangle_{flat}$ is invariant under all flat-virtual Reidemeister moves except for the first classical Reidemeister move.

Moreover, if for an oriented flat-virtual knot diagram ${\bar K}$ we define the \emph{flat-virtual Jones polynomial}

$$J_{f}({\bar K})=(-a)^{-3w({\bar K})} \langle K\rangle_{flat},$$
where $K$ is the underlying non-oriented diagram,

then $J_{f}$ is an invariant of flat-virtual links.

\end{theorem} 

\begin{proof}
The proof repeats the reasonings for the classical or virtual Kauffman bracket~\cite{VirtualKnots}. The skein relations implies invariance of the bracket under second and third Reidemeister move on classical crossings, and show that first Reidemeister move multiplies the bracket by $(-a)^{\pm 3}$. Reidemeister moves on flat and virtual crossings commute with the bracket. Being applied to a mixed classical-virtual of classical-flat third Reidemeister move, the bracket produces two virtual or flat second Reidemeister moves~\cite[Figure 15]{VirtualKnots}.  
\end{proof}


Let us consider a simple example obtained from the Whitehead link, see Fig.~\ref{fig:whitehead_link} top left. The complement to the trivial (red) component is homeomorphic to the thickened cylinder. Hence, the other component defines a knot $K$ in the cylinder (Fig.~\ref{fig:whitehead_link} top right). For $d=2$, the diagram of the knot $K$ has two pairs of equivalent points $A$ and $B$ (Fig.~\ref{fig:whitehead_link} bottom right). Hence, the the corresponding flat-virtual knot $\phi_2(K)$ (Fig.~\ref{fig:whitehead_link} bottom left) has two flat crossings. 
 
\begin{figure}[h]
\centering\includegraphics[width=0.5\textwidth]{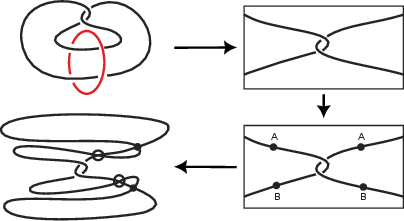}
\caption{Whitehead link (top left), the corresponding cylinder knot $K$ (top right), and the flat-virtual knot $\phi_2(K)$ (bottom left). The cylinder is presented as a rectangle with the vertical sides identified.}\label{fig:whitehead_link}
\end{figure}

Let us calculate the flat-virtual Jones polynomial of the flat-virtual knot $\phi_2(K)$ (Fig.~\ref{fig:whitehead_bracket}). 

\begin{figure}[h]
\centering\includegraphics[width=0.8\textwidth]{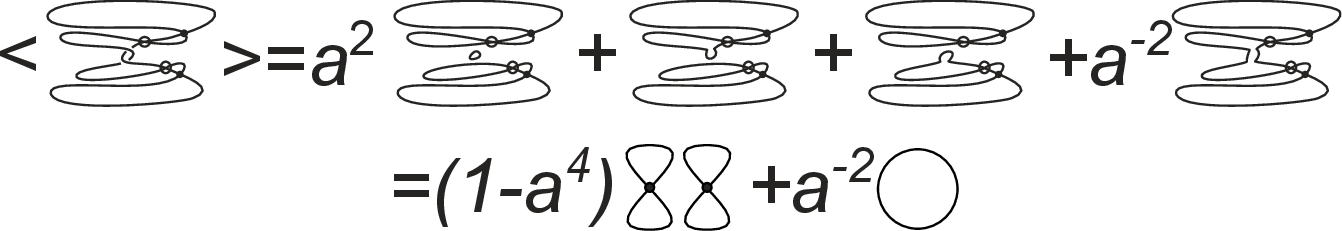}
\caption{Jones polynomial of $\phi_2(K)$}\label{fig:whitehead_bracket}
\end{figure}

Then
\[
J_f(\phi_2(K))=(-a)^{-6}\langle \phi_2(K)\rangle=(a^{-6}-a^{-2})
\raisebox{-0.29\height}{\includegraphics[width=2em]{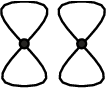}}+a^{-8}\bigcirc.
\]

Since $J_f(\phi_2(K))\ne\bigcirc$, we conclude that $\phi_2(K)$ is a non-trivial flat-virtual knot.

\section{Cobordisms}

This section contains a sketch of a future research. Here we mention crucial points:

\begin{enumerate}

\item Sliceness in classical knot theory plays a crucial role.
There are two types of sliceness: {\bf locally flat} sliceness and {\bf smooth} sliceness.

\item The difference between these two notions can lead to some very strong consequences,
for example, knots which are topologically slice but not smoothly slice, allow one to construct
exotic structures on $\R^{4}$.

\item Sliceness obstructions and cobordism bounds in classical knot theory often
require very powerful techniques (for example, Rasmussen invariant \cite{Ras}), which, in many
cases is difficult to calculate (one requires spectral sequences to calculate it in the general cases).

\item On the other hand, there are numerous ways of getting ``easy'' sliceness obstructions for
virtual knots. For example, in joint paper by D.A.Fedoseev and the first named author, \cite{FM}, it is proved that
under some circumstances {\em if a free knot is (smoothly) slice then it is elementarily slice} (by a slicing disc
without triple points and cusps). The condition of being ```elementarily slice'' is easy to check.

\item The purpose of this chapter is to formulate a general theorem saying that
``if a knot (link) in the full torus is slice then its image (a flat-virtual link) is slice
(the definition is given below).
Potentially, one can get ``easy'' obstructions from theory of virtual knots and free knots (say, \cite{FM}).

\end{enumerate}

The main result of the present section is the following

\begin{theorem}
If a link $K$ in the full torus $S^{1}\times D^{2}$ is smoothly slice in
 $S^{1}\times D^{3}$, then for each $d$ the flat-virtual link
 $\phi_{d}(K)$ is slice.
\end{theorem}

Actually, the proof will follow from the definitions once the theorem and all definitions
are precisely formulated.

\begin{dfn}
A
$k$-component link $L=L_{1}\sqcup \cdots \sqcup L_{k}\in S^{1}\times D^{2}$
is {\em smoothly slice}, if there are smoth discs $D^{2}_{1},
\cdots, D^{2}_{k}$, properly disjointly embedded in
 $S^{1}\times D^{2}\times [0,\infty)$ in such a way that
$\partial D_{j}= L_{j}\subset S^{1}\times D^{2}\times\{0\}\equiv S^{1}\times D^{2}$.
\end{dfn}

Now, we pass to the  sliceness of flat-virtual links.
Let $L$ be a flat-virtual link given by its
diagram $L_{1} \sqcup \cdots \sqcup L_{k}$ on the plane
$\R^{2}$. From the definition of sliceness of $L$
we shall need a ``two-dimensional diagram''
for the set of slicing discs $D_{1},\cdots, D_{k}$ in
the half-space $\R^{2}\times [0,\infty)$,
with the initial diagram on the boundary $\R^{2}\times \{0\}\equiv\R^{2}$.

\begin{dfn}
A {\em two-dimensional diagram} is a set of  properly disjointly immersed
discs $\sqcup D_{j} \to \R^{2}\times [0,\infty)$ endowed with a certain additional
structure.
\end{dfn}

\subsection{Generic immersions and structure}

A generic map  $f:\sqcup D_{j}\to \R^{2}\times [0,\infty)$ should have the following types
of points having more than one preimage:
\begin{enumerate}
\item double lines,
\item cusps,
\item triple points.
\end{enumerate}

Herewith each double line cares the following information: it is either {\em classical} (in
this case we indicate which sheet is over, the other one being under) or
{\em virtual} or {\em flat}.

Besides, we require that the double line approaching the cusp is
either classical or virtual
\footnote{Here we refer to standard 2-knot diagrams}.

Double lines passing through triple points should be of the following types:
three classical ones (standard), two flat and one classical, three flat ones (only for
$d>2$) or two virtual ones and one more double line (of
any type).

\subsection{
The proof sketch of the theorem about cobordism}

If there is a set of slicing discs in $S^{1}\times D^{2}\times [0,\infty)=
S^{1}_{\phi}\times I_{s}\times I_{t}\times [0,\infty)$,
then it suffices only to take a generic projection to $S_{\phi}^{1}\times I_{s}\times[0,\infty)_{u}$.

After that one should form flat double lines by identifying those
pairs of points which share $s,u$ coordinates and have coordinates $\phi$
which differs by $\frac{2\pi j}{d},j\in \N$.
Thus, we get flat double lines.

Virtual double lines naturally appear
when projecting
$S_{\phi}\times I_{s}\times [0,\infty)_{u}$ onto $\R^{2}\times [0,\infty)_{u}$
in those places where the projection has at least two preimages.

\section{Further directions and unsolved problems}

The motivation for this work is to construct a map from classical knots to virtual ones. In the paper~\cite{MN1} we managed to construct maps from classical braids to virtual braids. Extension of those maps to ones compatible with the Markov moves would give the desired map classical knots (considered as closures of braids) to virtual knots.

A further generalization of flat-virtual knot theory is to consider crossing with various group or homotopical labels. A possible variant of such a theory is described in the appendix.

\pagebreak
\appendix

\section{Multi-flat knots}

In this section we define a generalization of flat-virtual knots. The multi-flat knots considered below have crossings of several types which are ordered. The classical crossings have the lowest type, the crossings of other types are flat. This hierarchy of crossings appears naturally when one considers link diagrams in surfaces and (branched) covering maps between them. Self-intersections of the projection of a link diagram by a covering map are treated as flat crossings of a new type.   

\subsection{Definition}

\begin{definition}
    Let $S$ be an oriented connected compact surface. Let $k\in\mathbb Z_{\ge 0}$. A \emph{diagram with flat crossings of $k$ types} (or \emph{$k$-flat link diagram}) is a $4$-valent graph $D$ embedded in $S$ whose vertices either have a classical undercrossing-overcrossing structure (Fig.~\ref{fig:crossing_types} left) or are marked with numbers $i$, $1\le i\le k$ (Fig.~\ref{fig:crossing_types} right). The number $i$ is called the \emph{type} of the crossing; the type of the classical crossings is assumed to be $0$. Denote the set of $k$-flat diagrams in $S$ by $\mathcal D_k(S)$.

\begin{figure}[h]
\centering\includegraphics[width=0.2\textwidth]{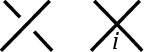}
\caption{A classical crossing and a flat crossing of type $i$ ($i$-crossing)}\label{fig:crossing_types}
\end{figure}

    Consider the classical second and third Reidemeister moves $\Omega_2(0)$ and $\Omega_3(0,0)$ (Fig.~\ref{fig:reidemeister_moves} left) and analogous flat moves $\Omega_2(i)$, $\Omega_3(0,j)$ and  $\Omega_3(i,j)$,  $1\le i\le j\le k$ (Fig.~\ref{fig:reidemeister_moves} right).

\begin{figure}[h]
\centering\includegraphics[width=0.7\textwidth]{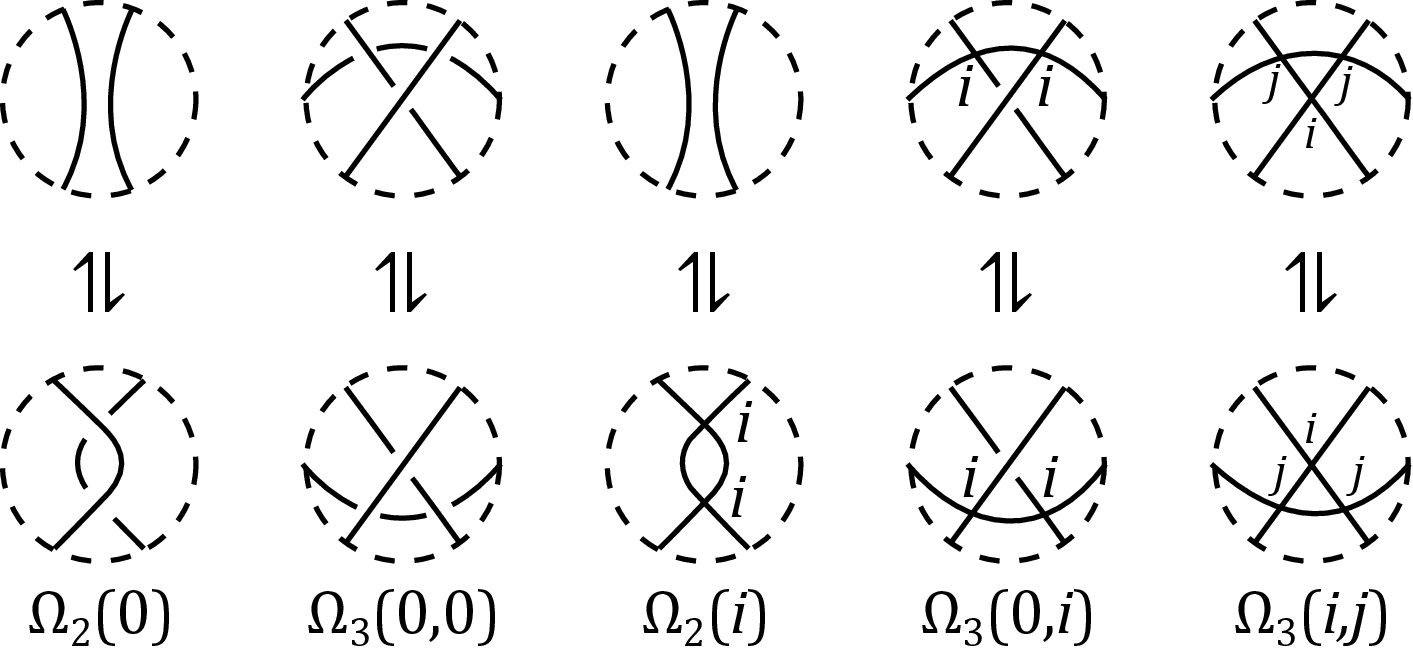}
\caption{Reidemeister moves}\label{fig:reidemeister_moves}
\end{figure}

    Define the move $\Omega_1^d(i)$, $0\le i\le k$, as shown in Fig.~\ref{fig:reidemeister_move1}. In the case $i=0$, we consider all variants of the undercrossing-overcrossing structure of the crossings in the move.

\begin{figure}
\centering\includegraphics[width=0.3\textwidth]{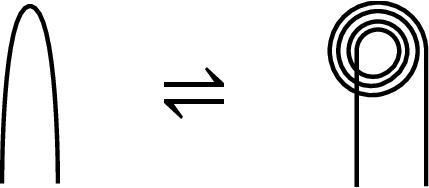}
\caption{Move $\Omega_1^d(i)$ (case $d=3$). The crossings are of type $i$}\label{fig:reidemeister_move1}
\end{figure}

    Fix tuples $\mathbf d=(d_0,d_1,\dots,d_k)\in\mathbb Z_{\ge 0}^{k+1}$ and $\mathbf\epsilon=(\epsilon_0,\epsilon_1,\dots,\epsilon_k)\in\{0,1\}^{k+1}$. Define the set of \emph{multi-flat links of type $(\mathbf d, \mathbf\epsilon)$} as the set $\mathcal L_{\mathbf d, \mathbf\epsilon}(S)$ of equivalence classes of $k$-flat diagrams in the surface $S$ modulo diagram isotopies and the moves 
    \begin{enumerate}
        \item first Reidemeister moves $\Omega_1^{d_i}(i)$, $0\le i\le k$,
        \item second Reidemeister moves $\Omega_2(i)$, $0\le i\le k$,
        \item third Reidemeister moves $\Omega_3(i,i)$ for $i$ such that $\epsilon_i=1$,
        \item mixed third Reidemeister moves $\Omega_3(i,j)$, $0\le i<j\le k$.
    \end{enumerate}

    When $\mathbf d=(1,\dots,1)$ and $\mathbf\epsilon=(1,\dots,1)$ we denote $\mathcal L_{\mathbf d, \mathbf\epsilon}(S)=\mathcal L_k(S)$ and call its elements \emph{$k$-flat links}.
\end{definition}

If we forget about under-overcrossing structure of the classical crossings and treat them as flat crossings, we get the theory of \emph{flat multi-flat links} $\bar{\mathcal L}_{\mathbf d, \mathbf\epsilon}(S)$ which has $k+1$ types of flat crossings. In particular, we get \emph{flat $k$-flat links} $\bar{\mathcal L}_{k}(S)=\bar{\mathcal L}_{(1,\dots,1), (1,\dots,1)}(S)$.
Below we denote the set of flat $k$-flat diagrams in $S$ by $\bar{\mathcal D}_k(S)$.

\begin{example}
\begin{itemize}
    \item The $0$-flat links $\mathcal L_0(S^2)=\mathcal L_{(1),(1)}(S^2)$ in the sphere are classical links.
    \item The $1$-flat links $\mathcal L_1(S^2)=\mathcal L_{(1,1),(1,1)}(S^2)$ in the sphere are virtual links. The $0$-crossings are classical, and $1$-crossings are virtual.
    \item The $2$-flat links $\mathcal L_2(S^2)=\mathcal L_{(1,1,1),(1,1,1)}(S^2)$ are flat-virtual links~\cite{MN2}. The $0$-crossings are classical, $1$-crossings are flat, and $2$-crossings are virtual crossings.    
    \item The multi-flat links $\mathcal L_{(0),(1)}(S^2)$ are regular classical links. 
    \item The multi-flat links $\mathcal L_{(1),(0)}(S)$ are doodles in the surface $S$.
\end{itemize}
\end{example}

\begin{remark}
    Multi-flat links can be considered as special cases of generalized knot theories~\cite{BF1,BF2,F} defined by R. Fenn.  A \emph{generalized knot theory} is defined on diagrams whose crossings are tagged by a type (an element of some set $T$ with an involution; the involution indicates the sign of crossings).  The moves of the theory are Reidemeister moves allowed for certain combinations of crossing types.

    For multi-flat knots, the set of types is $T=\{0_+,0_-,1,\dots,k\}$ with the involution $\bar 0_\pm=0_\mp$, $\bar i=i$, $1\le i\le k$.
\end{remark}

\subsection{Quotient map}

\begin{proposition}
    For any tuples $\mathbf\epsilon$, $\mathbf\epsilon'$ in $\{0,1\}^{k+1}$ such that $\epsilon'_i\ge\epsilon_i$, $0\le i\le k$, and tuples $\mathbf d$, $\mathbf d'$ in $\mathbb Z_{\ge 0}^{k+1}$ such that $d'_i\mid d_i$, $0\le i\le k$, the identity map on the $k$-flat diagrams $D\in\mathcal D_k(S)$ induces a natural projection $p\colon \mathcal L_{\mathbf d, \mathbf\epsilon}(S)\to\mathcal L_{\mathbf d', \mathbf\epsilon'}(S)$.
\end{proposition}
\begin{proof}
    Indeed, the multi-flat moves of type $(\mathbf d, \mathbf\epsilon)$ are expressed by the moves of type $(\mathbf d', \mathbf\epsilon')$.
\end{proof}

Analogously, one defines a quotient map $p\colon \bar{\mathcal L}_{\mathbf d, \mathbf\epsilon}(S)\to\bar{\mathcal L}_{\mathbf d', \mathbf\epsilon'}(S)$.

In particular, for any tuples $\mathbf d\in\mathbb Z_{\ge 0}^{k+1}$ and $\mathbf\epsilon\in\{0,1\}^{k+1}$ there are quotient maps $p\colon \mathcal L_{\mathbf d, \mathbf\epsilon}(S)\to\mathcal L_{k}(S)$ and $p\colon \bar{\mathcal L}_{\mathbf d, \mathbf\epsilon}(S)\to\bar{\mathcal L}_{k}(S)$.

\subsection{Inclusion map}

\begin{proposition}
    Let $\mathbf d=(d_0,d_1,\dots,d_k)$, $\mathbf\epsilon=(\epsilon_0,\epsilon_1,\dots,\epsilon_k)$, $\mathbf d'=(d'_0,d'_1,\dots,d'_l)$, $\mathbf\epsilon'=(\epsilon'_0,\epsilon'_1,\dots,\epsilon'_l)$, and $\sigma=(\sigma_0,\dots,\sigma_k)\subset\mathbb Z_{\le 0}^{k+1}$ be such that
\begin{itemize}
    \item $k\le l$;
    \item $0\le\sigma_0<\sigma_1<\cdots<\sigma_k\le l$;
    \item for any $i=0,\dots,k$  $\epsilon'_{\sigma_i}=\epsilon_i$ and $d'_{\sigma_i}=d_i$.
\end{itemize}
Then the map $\iota_\sigma\colon \mathcal D_k(S)\to\mathcal D_l(S)$ which changes the type of each $i$-crossing in a diagram $D\in \mathcal D_k(S)$ to $\sigma_i$, $i=0,\dots,k$, induces a well-defined map $\iota_\sigma\colon \mathcal L_{\mathbf d, \mathbf\epsilon}(S)\to\mathcal L_{\mathbf d', \mathbf\epsilon'}(S)$.
\end{proposition}
\begin{proof}
    Indeed, the inclusion map juxtaposes the move $\Omega_2(i)$ with $\Omega_2(\sigma_i)$, the move $\Omega_3(i,j)$ with $\Omega_3(\sigma_i,\sigma_j)$ and the move $\Omega_1^d(i)$ with $\Omega_1^d(\sigma_i)$, which are the multi-flat moves of type $(\mathbf d', \mathbf\epsilon')$ by the conditions on $\mathbf d', \mathbf\epsilon'$.
\end{proof}

Analogously, one defines an inclusion map $\iota_\sigma\colon \bar{\mathcal L}_{\mathbf d, \mathbf\epsilon}(S)\to\bar{\mathcal L}_{\mathbf d', \mathbf\epsilon'}(S)$ (or to $\mathcal L_{\mathbf d', \mathbf\epsilon'}(S)$ if $\sigma_0>0$).

In particular, for any sequence $\sigma=(\sigma_0,\dots,\sigma_k)$, $0\le\sigma_0<\sigma_1<\cdots<\sigma_k\le l$, we have inclusion maps $\iota_\sigma\colon \mathcal L_{k}(S)\to\mathcal L_{l}(S)$ and $\iota_\sigma\colon \bar{\mathcal L}_{k}(S)\to\bar{\mathcal L}_{l}(S)$.

\begin{example}\label{ex:2-flat-8-knot}
Consider an eight-knot $K\in\mathcal L_2$ (Fig.~\ref{fig:inclusion_map} left). The result of the inclusion map $\iota_\sigma\colon \mathcal L_{2}\to\mathcal L_{5}$ given by $\sigma=(2,3,5)$ applied to $K$ is shown in Fig.~\ref{fig:inclusion_map} right.
\begin{figure}[h]
\centering\includegraphics[width=0.15\textwidth]{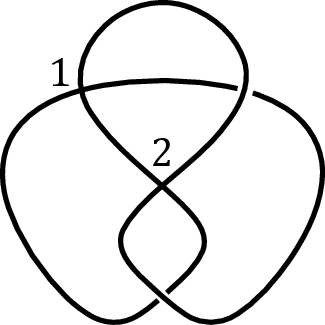}\qquad
\includegraphics[width=0.15\textwidth]{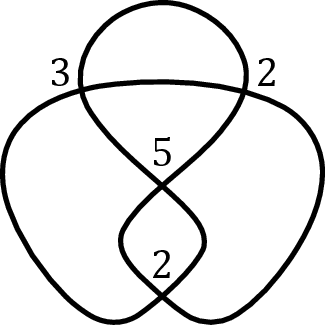}
\caption{A $2$-flat eight knot $K$ and its image under the inclusion map $\iota_{(2,3,5)}$}\label{fig:inclusion_map}
\end{figure}

\end{example}

\subsection{Merging map}

Let $k\in\mathbb N$ and $1\le l<k$. Consider the map $\mu_l$ which replaces the type of all the $i$-crossings of a diagram $D\in\mathcal D_k(S)$, $l<i\le k$, by $i-1$ (hence, the $l$-crossings and $(l+1)$-crossings become $l$-crossings).

\begin{proposition}
    The map $\mu_l$ induces a well-defined map $\mu_l\colon\mathcal L_{\mathbf d, \mathbf\epsilon}(S)\to\mathcal L_{\mathbf d', \mathbf\epsilon'}(S)$ where $\mathbf d'=(d'_0,d'_1,\dots,d'_{k-1})$, $\mathbf\epsilon'=(\epsilon'_0,\epsilon'_1,\dots,\epsilon'_{k-1})$ with
\[
d'_i=\left\{\begin{array}{lc}
    d_i, & i<l, \\
    gcd(d_l,d_{l+1}), & i=l,\\
    d_{i+1}, & i>l,
\end{array}\right.
\qquad
\epsilon'_i=\left\{\begin{array}{lc}
    \epsilon_i, & i<l, \\
    1, & i=l,\\
    \epsilon_{i+1}, & i>l.
\end{array}\right.
\]
\end{proposition}

\begin{proof}
    Indeed, the merging map $\mu_l$ converts the multi-flat moves of type $(\mathbf d, \mathbf\epsilon)$ to multi-flat moves of type $(\mathbf d', \mathbf\epsilon')$. For example, the moves $\Omega_1^{d_l}(l)$ and $\Omega_1^{d_{l+1}}(l+1)$ turn into moves $\Omega_1^{d_l}(l)$ and $\Omega_1^{d_{l+1}}(l)$ which are equivalent to the move $\Omega_1^{gcd(d_l,d_{l+1})}(l)$. The move $\Omega_3(l,l+1)$ turns into $\Omega_3(l,l)$.
\end{proof}

Analogously, one defines the merging map $\mu_0\colon {\mathcal L}_{\mathbf d, \mathbf\epsilon}(S)\to\bar{\mathcal L}_{\mathbf d', \mathbf\epsilon'}(S)$ which corresponds to the case $l=0$.

In particular, we have merging maps $\mu_l\colon \mathcal L_{k}(S)\to\mathcal L_{k-1}(S)$, $1\le l< k$, and $\mu_0\colon \mathcal L_{k}(S)\to\bar{\mathcal L}_{k-1}(S)$.

\begin{example}
Consider the $2$-flat eight-knot $K$ from Example~\ref{ex:2-flat-8-knot}. The results of merging of $K$ is shown in Fig.~\ref{fig:merging_map}. 
\begin{figure}[h]
\centering\includegraphics[width=0.15\textwidth]{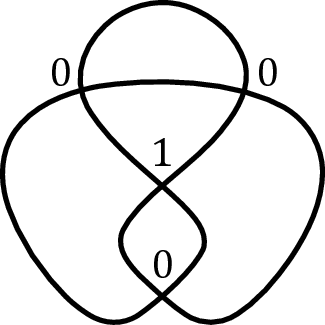}\qquad
\includegraphics[width=0.15\textwidth]{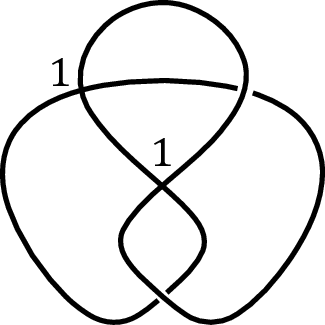}
\caption{The mergings $\mu_0(K)$ and $\mu_1(K)$ of the $2$-flat eight knot $K$}\label{fig:merging_map}
\end{figure}
    
\end{example}

\subsection{Covering map}

Let $f\colon S_1\to S_2$ a branched covering of degree $p$ and ramification indices $r_1,\dots, r_l$ at the ramification points. Consider a diagram $D\in\mathcal D_k(S_1)$ in general position with respect to $f$, i.e. the image $D'=f(D)$ has only double self-intersection points distinct from the crossings of $D$ and the branching points. Mark the double points in $D'$ as $(k+1)$-crossings. Denote the resulting diagram by $f_*(D)\in\mathcal D_{k+1}(S_2)$.

\begin{proposition}[Covering map]
    The map $D\mapsto f_*(D)$ induces a well-defined map 
\[
f_*\colon \mathcal L_{\mathbf d, \mathbf\epsilon}(S_1)\to \mathcal L_{\mathbf d', \mathbf\epsilon'}(S_2)
\]
where $
\mathbf d'=(d_0,\dots,d_k,d'_{k+1}),\quad \mathbf \epsilon'=(\epsilon_0,\dots,\epsilon_k,\epsilon'_{k+1})
$ with 
\[
d'_{k+1}=gcd(r_1-1,\dots,r_l-1),\quad
\epsilon'_{k+1}=\left\{\begin{array}{cl}
    0, & p=2, \\
    1, & p>2.
\end{array}\right.
\]
\end{proposition}

In particular, we have covering maps $f_*\colon \mathcal L_{k}(S_1)\to \mathcal L_{k+1}(S_2)$.

\begin{proof}
    If diagrams $D_1,D_2\in \mathcal D_k(S_1)$ are connected by a multi-flat Reidemeister move, we can assume it occurs in a small disk $B\subset S_1$.
    Then $f_*(D_1)$ and $f_*(D_2)$ are connected by a multi-flat Reidemeister move of the same type as the one between $D_1$ and $D_2$.

    An isotopy between diagrams $D_1$ and $D_2$ in $S_1$ can be split into sequence of local isotopies. After applying the map $f$ to a local isotopy, we can get either a second Reidemeister move $\Omega_2(k+1)$ or a third Reidemeister move $\Omega_3(i,k+1)$, $0\le i\le k$, or a third Reidemeister move $\Omega_3(k+1,k+1)$ (when arcs in three layers of the covering form a triple point after projection to $S_2$; this case requires $p\ge 3$). A move of an arc over a ramification point in $S_1$ of index $r$ generates a first Reidemeister move $\Omega_1^{r-1}(k+1)$ in $S_2$ (Fig.~\ref{fig:projection_moves}).

\begin{figure}[h]
\centering\includegraphics[width=0.7\textwidth]{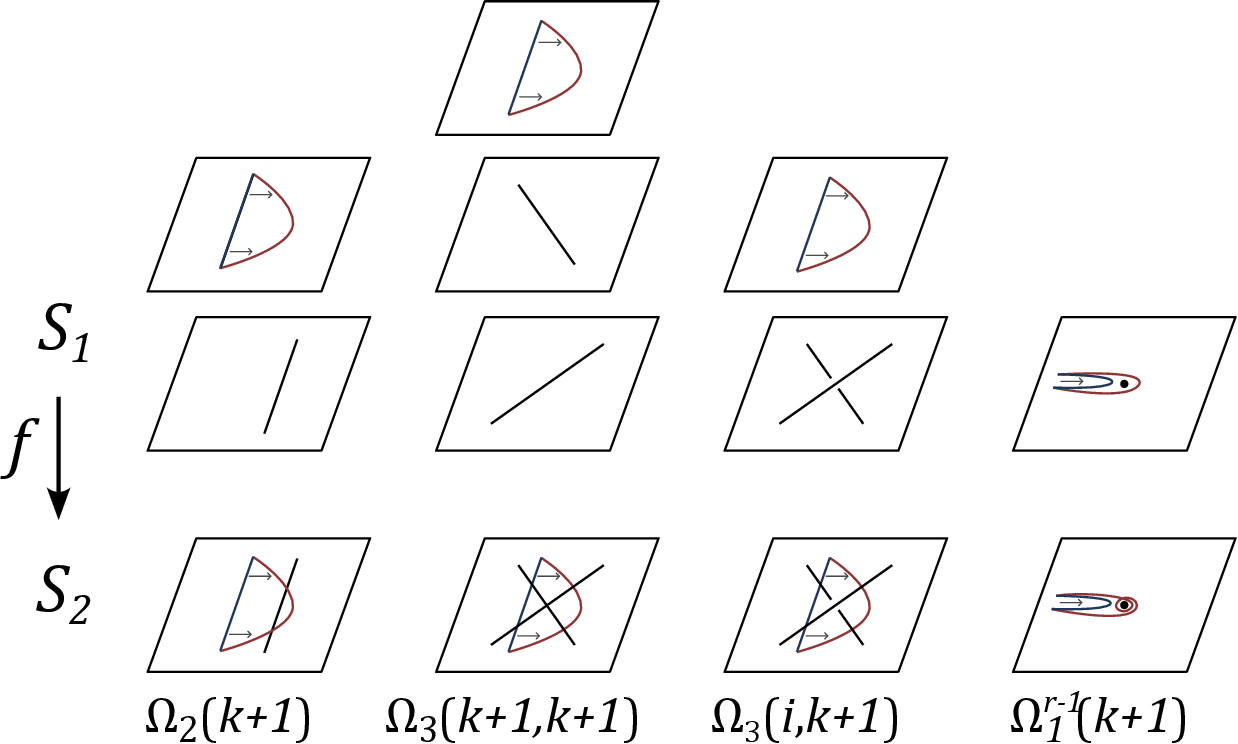}
\caption{Moves induced by projection of an isotopy}\label{fig:projection_moves}
\end{figure}

    Thus, for any equivalent diagrams $D_1$ and $D_2$ in $\mathcal L_{\mathbf d, \mathbf \epsilon}(S_1)$, the diagrams $f_*(D_1)$ and $f_*(D_2)$ are equivalent in $\mathcal L_{\mathbf d', \mathbf \epsilon'}(S_2)$.
\end{proof}

\begin{example}
    Consider a $0$-flat knot $K$ in an annulus $S$ (Fig.~\ref{fig:mf_covering_map} left). Let $f\colon S\to S$ be the regular $2$-fold covering. Then $f_*(K)\in \mathcal L_1(S)$ is a $1$-flat knot with one classical and two flat crossings (Fig.~\ref{fig:mf_covering_map} right).
\begin{figure}[h]
\centering\includegraphics[width=0.4\textwidth]{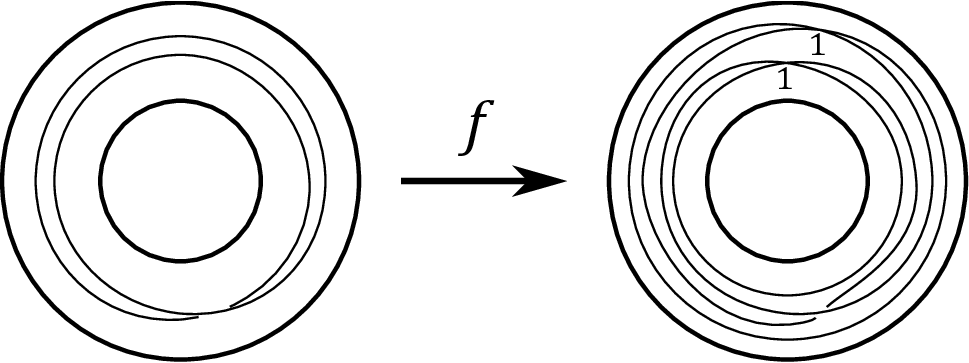}
\caption{The covering map on a $0$-flat knot}\label{fig:mf_covering_map}
\end{figure}
\end{example}

We call a link $L\in \mathcal L_{\mathbf d, \mathbf\epsilon}(S)$ a \emph{classical link} if there is a diagram $D$ of $L$ which has only classical crossings and is contained in a disk $B\subset S$.

\begin{proposition}
    For a classical link $L$ in the surface $S_1$, its image $f_*(L)$ under any covering map $f\colon S_1\to S_2$ is a classical link in $S_2$. 
\end{proposition}
\begin{proof}
    Indeed, we can suppose that the restriction $f|_B$ is a homeomorphism. Hence, $f_*(D)\subset f(B)$ is a classical diagram.
\end{proof}

\subsection{Picture-valued Jones polynomial}

Denote $A=\mathbb Z[a,a^{-1}]$. Given tuples $\mathbf d=(d_0,d_1,\dots,d_k)$, $\mathbf\epsilon=(\epsilon_0,\epsilon_1,\dots,\epsilon_k)$, define the \emph{skein module} $\mathcal S_{\mathbf d, \mathbf\epsilon}(S)$ as the quotient $A$-module of the free module generated by the set of isotopy classes of diagrams $\mathcal D_{k}(S)$ modulo the relations
\begin{itemize}
    \item $\skcrossr=a\cdot\skcrv+a^{-1}\cdot\skcrh$;
    \item $L\sqcup\bigcirc=(-a^2-a^{-2})L$;
    \item Reidemeister moves of $\mathcal L_{\mathbf d, \mathbf\epsilon}(S)$ on $i$-crossings, $i>0$.
\end{itemize}

\begin{proposition}
Let $\mathbf d'=(d_1,\dots,d_k)$, $\mathbf\epsilon'=(\epsilon_1,\dots,\epsilon_k)$.
    There are isomorphism
\[
\mathcal S_{\mathbf d, \mathbf\epsilon}(S)\simeq A[\bar{\mathcal L}_{\mathbf d', \mathbf\epsilon'}(S)]/(L\sqcup\bigcirc=(-a^2-a^{-2})L)\simeq A[\bar{\mathcal L}^{red}_{\mathbf d', \mathbf\epsilon'}(S)]
\]
induced by the inclusion map $\iota_\sigma\colon\bar{\mathcal L}_{\mathbf d', \mathbf\epsilon'}(S)\to\mathcal L_{\mathbf d, \mathbf\epsilon}(S)$ where 
$\sigma=(1,\dots, k)$, and $\bar{\mathcal L}^{red}_{\mathbf d', \mathbf\epsilon'}(S)\subset \bar{\mathcal L}_{\mathbf d', \mathbf\epsilon'}(S)$ is the set of multi-flat links without trivial components.
\end{proposition}

\begin{proof}
    Consider the submodule $\mathcal S^0_{\mathbf d, \mathbf\epsilon}(S)\subset \mathcal S_{\mathbf d, \mathbf\epsilon}(S)$ which is spanned by diagrams without classical crossings. An isomorphism between between $\mathcal S^0_{\mathbf d, \mathbf\epsilon}(S)$ and $\mathcal S_{\mathbf d, \mathbf\epsilon}(S)$ is established by the inclusion and the Kauffman bracket formula $\langle\cdot\rangle\colon\mathcal S_{\mathbf d, \mathbf\epsilon}(S)\to \mathcal S^0_{\mathbf d, \mathbf\epsilon}(S)$
\begin{equation*}
    \langle D\rangle =\sum_{s} a^{\alpha(s)-\beta(s)} D_s.
\end{equation*}
The inclusion map $\iota_\sigma$ identifies $A[\bar{\mathcal L}_{\mathbf d', \mathbf\epsilon'}(S)]/(L\sqcup\bigcirc=(-a^2-a^{-2})L)$ with $\mathcal S^0_{\mathbf d, \mathbf\epsilon}(S)$.
\end{proof}

For a diagram $D\in\mathcal D_k(S)$ denote its class in $\mathcal S_{\mathbf d, \mathbf\epsilon}(S)$ by $\langle D\rangle$.

\begin{proposition}
    The formula $X(D)=(-a)^{-3w(D)}\langle D\rangle$  induces a well-defined map
    $$X\colon\mathcal L_{\mathbf d, \mathbf\epsilon}(S)\to A[\bar{\mathcal L}^{red}_{\mathbf d', \mathbf\epsilon'}(S)].$$ Here $w(D)$ is the writhe of the diagram $D$ (the sum of the signs of the classical crossings).
\end{proposition}

\begin{proof}
    The standard arguments for the Kauffman bracket implies that the bracket $D\mapsto\langle D\rangle$ is invariant is invariant under classical second and third Reidemeister moves. A first Reidemeister move which creates a crossing $c$, multiplies the bracket by $(-a)^{3sgn(c)}$, hence, $X(D)$ is an invariant.
\end{proof}

The invariant $X(D)$ is called the \emph{(picture-valued) Jones polynomial} of the multi-flat link $D$.

In particular, we have Jones polynomial $X\colon\mathcal L_{k}(S)\to A[\bar{\mathcal L}^{red}_{k-1}(S)]$.

\begin{example}
    Consider the $2$-flat eight knot $K$ (Fig.~\ref{fig:inclusion_map} left). Its bracket $\langle K\rangle$ is shown in Fig.~\ref{fig:eight_knot_mf_bracket}. Since $w(K)=0$, we have $X(K)=\langle K\rangle$.

\begin{figure}[h]
\centering\includegraphics[width=0.7\textwidth]{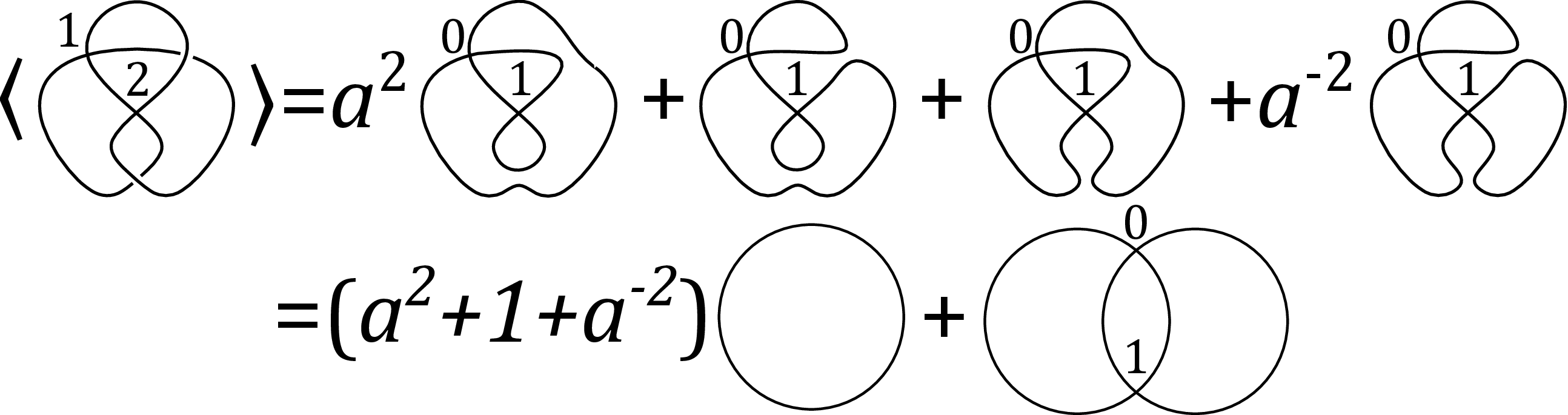}
\caption{The bracket of the $2$-flat eight knot $K$}\label{fig:eight_knot_mf_bracket}
\end{figure}
\end{example}

By specifying $a=-1$, one defines the \emph{flat skein module} $\bar{\mathcal S}_{\mathbf d, \mathbf\epsilon}(S)$ as the quotient of the free $\mathbb Z$-module generated by the set $\bar{\mathcal L}_{\mathbf d, \mathbf\epsilon}(S)$ modulo the relations
\begin{itemize}
    \item $\skcrf=-\skcrv-\skcrh$;
    \item $L\sqcup\bigcirc=-2L$.    
\end{itemize}

\begin{proposition}
Let $\mathbf d'=(d_1,\dots,d_k)$, $\mathbf\epsilon'=(\epsilon_1,\dots,\epsilon_k)$.

\begin{enumerate}
    \item  There are isomorphism
\[
\bar{\mathcal S}_{\mathbf d, \mathbf\epsilon}(S)\simeq \mathbb Z[\bar{\mathcal L}_{\mathbf d', \mathbf\epsilon'}(S)]/(L\sqcup\bigcirc=-2L)\simeq\mathbb Z[\bar{\mathcal L}^{red}_{\mathbf d', \mathbf\epsilon'}(S)];
\]
\item 
    The natural projection from $\bar{\mathcal D}_k(S)$ to $\bar{\mathcal S}_{\mathbf d, \mathbf\epsilon}(S)$ induces a well-defined map 
\[
X\colon\bar{\mathcal L}_{\mathbf d, \mathbf\epsilon}(S)\to\mathbb Z[\bar{\mathcal L}^{red}_{\mathbf d', \mathbf\epsilon'}(S)].
\]
\end{enumerate}
\end{proposition}

In particular, we have a flat Jones polynomial $X\colon\bar{\mathcal L}_{k}(S)\to \mathbb Z[\bar{\mathcal L}^{red}_{k-1}(S)]$.

\subsection{Biquandle}

\begin{definition}
   A \emph{$k$-flat biquandle} is a set $X$ with $k+2$ binary operations $\oast^0, \uast^0, \ast^1,\dots, \ast^k$ such that
\begin{enumerate}
    \item the operations $\oast^0, \uast^0$ define a biquandle structure on $X$, i.e.~\cite{EN}
    \begin{enumerate}
        \item $x\oast^0 x= x\uast^0 x$ for any $x\in X$;
        \item the maps $\cdot\oast^0 y$, $\cdot\uast^0 y\colon X\to X$ are bijections for any $y$; the map $S_0(x,y)=(x\oast^0 y, y\uast^0 x)$ is a bijection on $X\times X$;
        \item for any $x,y,z\in X$
        \begin{gather*}
            (x\oast^0 y)\oast^0(z\oast^0 y)=(x\oast^0 z)\oast^0(y\uast^0 z),\\
            (x\uast^0 y)\oast^0(z\uast^0 y)=(x\oast^0 z)\uast^0(y\oast^0 z),\\
            (x\uast^0 y)\uast^0(z\uast^0 y)=(x\uast^0 z)\uast^0(y\oast^0 z);
        \end{gather*}
    \end{enumerate}
   \item the operation $\ast^i$, $1\le i\le k$, defines a flat biquandle structure on $X$, i.e. 
    \begin{enumerate}
        \item the map $\cdot\ast^i y\colon X\to X$ is a bijection for any $y$; the map $S_i(x,y)=(x\ast^i y, y\ast^i x)$ is a bijection on $X\times X$;
        \item for any $x,y,z\in X$ we have $(x\ast^i y)\ast^i(z\ast^i y)=(x\ast^i z)\ast^i(y\ast^i z)$;
    \end{enumerate}
    \item for any $0\le i<j\le k$ and any $x,y,z\in X$ we have
        \begin{gather*}
            (x\oast^i y)\ast^j(z\ast^j y)=(x\ast^j z)\oast^i(y\ast^j z),\\
            (x\uast^i y)\ast^j(z\ast^j y)=(x\ast^j z)\uast^i(y\ast^j z),\\
            (x\ast^j y)\ast^j(z\oast^i y)=(x\ast^j z)\ast^j(y\uast^i z).
        \end{gather*}
        Here, for $i>0$ we assume $\uast^i=\oast^i=\ast^i$. 
\end{enumerate}
\end{definition}

Given a diagram $D\in \mathcal D_k(S)$ and a $k$-flat biquandle $X$, consider the set of colourings $Col_X(D)$ of the semiarcs of $D$ by elements of $X$ such that at each crossing the compatibility conditions in Fig.~\ref{fig:biquandle_operations} hold.

\begin{figure}
\centering\includegraphics[width=0.5\textwidth]{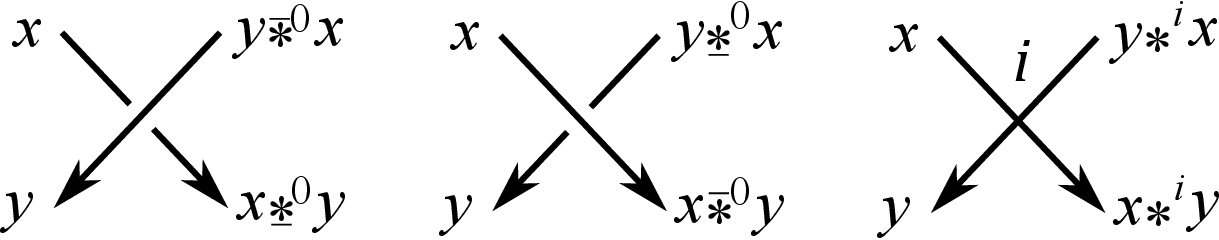}
\caption{Biquandle coloring}\label{fig:biquandle_operations}
\end{figure}

\begin{proposition}
    Let  $D$ and $D'$ be diagrams of one multi-flat link $L\in\mathcal L_{\mathbf d, \mathbf\epsilon}(S)$. Then there is a bijection between the sets of biquandle colourings $Col_X(D)$ and $Col_X(D')$.

    In particular, the number of biquandle colourings is an invariant of multi-flat links.
\end{proposition}

\begin{proof}
    The proof is standard check of bijection for the Reidemeister moves. For example, in the case of $\Omega_3(i,j)$ move, the 1-to-1 correspondence between the colourings follows from the third group of relation in the definition of multiflat biquandle (Fig.~\ref{fig:biquandle_r3ij}). The other moves are checked analogously. 

\begin{figure}[h]
\centering\includegraphics[width=0.6\textwidth]{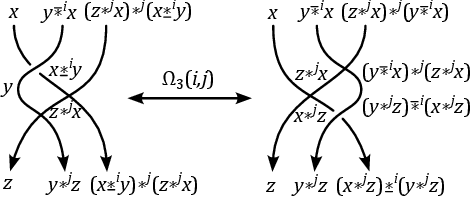}
\caption{The colourings in the case of $\Omega_3(i,j)$ move}\label{fig:biquandle_r3ij}
\end{figure}
\end{proof}

\begin{example}
    Let $X$ be a $k$-flat biquandle such that the element $x\ast^i y$, $x,y\in X$, $1\le i\le k$, depends only on $x$, i.e. $x\ast^i y=f_i(x)$ for some map $f_i\colon X\to X$. The conditions in the definition of $k$-flat biquandle mean that $(X,\uast^0,\oast^0)$ is a biquandle with $k$ commuting automorphisms $f_i$:
\[
f_i(x\uast^0 y)=f_i(x)\uast^0 f_i(y),\quad f_i(x\oast^0 y)=f_i(x)\oast^0 f_i(y),\quad
f_i\circ f_j= f_j\circ f_i
\]
for any $x,y\in X$, $1\le i,j\le k$.

For example, we can take a module $X$ over the ring $A_k=\mathbb Z[t,t^{-1},s_1,s_1^{-1},\dots, s_k,s_k^{-1}]$ with the operations
\[
x\uast^0 y=tx+(1-t)y,\quad x\oast^0 y = x,\quad x\ast^i y=f_i(x)=s_ix.  
\]
This $k$-flat biquandle is called a \emph{$k$-flat Alexander quandle}. 
\end{example}

Let $D$ is an oriented $k$-flat link diagram in $S$. The \emph{Alexander quandle} $AQ(D)$ is the module over the ring $A_k=\mathbb Z[t,t^{-1},s_1,s_1^{-1},\dots, s_k,s_k^{-1}]$ whose generators are the semiarcs of $D$ and relations are the coloring relations at the crossings (Fig.~\ref{fig:biquandle_operations}) with the operations $\uast^0, \oast^0, \ast^i$ defined in the previous example.

\begin{proposition}
 The Alexander quandle $AQ(D)$ is an invariant of multi-flat links $\mathcal L_{\mathbf d, \mathbf\epsilon}(S)$. 
\end{proposition}

As a consequence, the generator $\Delta(D)\in A_k$ of the Fitting ideal~\cite{L} $Fitt_0(AQ(D))$ is an invariant of multi-flat links (defined up to invertible elements in the ring $A_k$) which is called the \emph{Alexander polynomial} of the multi-flat link.

\begin{example}
    Consider the $2$-flat eight knot $K$ (Fig.~\ref{fig:inclusion_map} left). The Alexander quandle $AQ(K)$ is generated over $A_2$ by two generators $a,b$ corresponding to the outgoing undercrossing arcs (Fig.~\ref{fig:eight_knot_mf_aq}), and two relations corresponding to the crossings.
\begin{figure}[h]
\centering\includegraphics[width=0.15\textwidth]{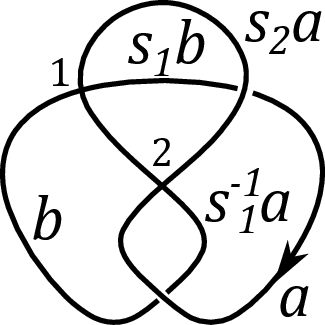}
\caption{The $2$-flat eight knot $K$}\label{fig:eight_knot_mf_aq}
\end{figure}
The relation matrix of $AQ(K)$ is
\[
M=\left(\begin{array}{cc}
    ts_1^{-1}+(1-t) & -1 \\
    t+(1-t)s_2 & -s_1
\end{array}\right).
\]
Then the Alexander polynomial of the $2$-flat eight knot is $\Delta(K)=det(M)=(s_2-s_1)(1-t)$. 
\end{example}

\subsection{Virtual multi-flat links}

\begin{definition}
    Let $S$ be an oriented connected compact surface. Let $k\in\mathbb Z_{\ge 0}$. A \emph{virtual $k$-flat link diagram} is a $4$-valent graph $D$ embedded in $S$ whose vertices are either classical or flat of type $i$, $1\le i\le k$, or virtual. Virtual crossings are drawn circled (Fig.~\ref{fig:detour_move}). Denote the set of virtual $k$-flat diagrams in $S$ by $\mathcal {VD}_k(S)$.

    Fix tuples $\mathbf d=(d_0,d_1,\dots,d_k)\in\mathbb Z_{\ge 0}^{k+1}$ and $\mathbf\epsilon=(\epsilon_0,\epsilon_1,\dots,\epsilon_k)\in\{0,1\}^{k+1}$. Define the set of \emph{virtual multi-flat links of type $(\mathbf d, \mathbf\epsilon)$} as the set $\mathcal {VL}_{\mathbf d, \mathbf\epsilon}(S)$ of equivalence classes of virtual $k$-flat diagrams in the surface $S$ modulo moves of multi-flat links of type $(\mathbf d, \mathbf\epsilon)$ and detour moves for virtual crossings (Fig.~\ref{fig:detour_move}).

\begin{figure}[h]
\centering\includegraphics[width=0.5\textwidth]{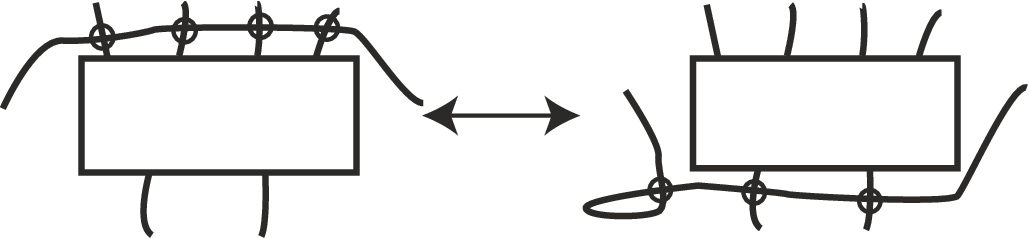}
\caption{Detour move}\label{fig:detour_move}
\end{figure}
\end{definition}

Consider by the following map $\nu\colon\mathcal D_{k+1}(S^2)\to \mathcal {VD}_k(S)$. For a $(k+1)$-flat diagram $D$ in the sphere, for any $i$-crossing of $D$, $0\le i\le k$, put a corresponding $i$-crossing somewhere in the surface $S$. Connect arbitrarily the ends of the crossings in $S$ in the same way as they are connected in the diagram $D$. Mark the intersection points of the connecting arcs as virtual crossings. The obtained virtual $k$-flat diagram $\nu(D)$ in $S$ is uniquely defined up to isotopy and detour moves.

\begin{proposition}
    Let $\mathbf d'=(d_0,d_1,\dots,d_k,1)$ and $\mathbf\epsilon'=(\epsilon_0,\epsilon_1,\dots,\epsilon_k,1)$. The map $\nu$ induces a well-defined bijection 
$\mathcal {VL}_{\mathbf d, \mathbf\epsilon}(S)\simeq \mathcal {L}_{\mathbf d', \mathbf\epsilon'}(S^2)$.
\end{proposition}

As a consequence, the theory of virtual multi-flat links does not depend on the surface $S$. 

\subsection{Alexander's and Markov's theorems}

Recall that a generalized knot theory~\cite{BF1} is \emph{regular} if it allows second Reidemeister move for each type of crossing; and it is \emph{normal} if there is a crossing type $x$ such that for any crossing type $y$ the third Reidemeister move on crossings of types $x,\bar x, y$ is allowed.

By definition we have the following statement.
\begin{proposition}\label{prop:f-flat_links_regular_normal}
    For any $\mathbf d=(d_0,d_1,\dots,d_k)\in\mathbb Z_{\ge 0}^{k+1}$ and $\mathbf\epsilon=(\epsilon_0,\epsilon_1,\dots,\epsilon_k)\in\{0,1\}^{k+1}$ the multi-flat link theory $\mathcal L_{\mathbf d, \mathbf\epsilon}(S)$ is regular. If $\epsilon_k=1$ then $\mathcal L_{\mathbf d, \mathbf\epsilon}(S)$ is normal.
\end{proposition}

A plane link diagram $D$  is called \emph{braided} if there exists a polar coordinate system in the plane such that the angular coordinate changes monotone on each component of $D$. One can consider a braided diagram as a closure of some braid.

\begin{theorem*}[Alexander's theorem \cite{BF1}]
In a regular knot theory, any diagram can be braided.
\end{theorem*}

\begin{theorem*}[Markov's theorem \cite{BF1}]
Suppose diagrams $D_1$ and $D_2$ are braided and define the same link in a
normal theory. Then they are related by a sequence of Reidemeister moves in which all the intermediate diagrams are braided.
\end{theorem*}

By Proposition~\ref{prop:f-flat_links_regular_normal}, we have Alexander's and Markov's theorems for multi-flat links.
\begin{corollary}
Let $\mathbf d=(d_0,d_1,\dots,d_k)\in\mathbb Z_{\ge 0}^{k+1}$ and $\mathbf\epsilon=(\epsilon_0,\epsilon_1,\dots,\epsilon_k)\in\{0,1\}^{k+1}$.
    \begin{enumerate}
        \item Any plane $k$-flat diagram can be braided in $\mathcal L_{\mathbf d, \mathbf\epsilon}(S^2)$;
        \item If $\epsilon_k=1$ then any two braided $k$-flat diagrams which are equivalent in $\mathcal L_{\mathbf d, \mathbf\epsilon}(S^2)$, are related in $\mathcal L_{\mathbf d, \mathbf\epsilon}(S^2)$ by a sequence of Reidemeister moves in which all the intermediate diagrams are braided. 
    \end{enumerate}
\end{corollary}

\end{document}